\documentclass[APA,Times1COL]{WileyNJDv5} 

\articletype{Original Article}%

\received{Date Month Year}
\revised{Date Month Year}
\accepted{Date Month Year}
\journal{Stat}
\volume{00}
\copyyear{2024}
\startpage{1}

\usepackage{orcidlink}
\usepackage{natbib}
\newcommand{\etr}{\mathop{\mathrm{etr}}\hskip1pt}
\newcommand{\EE}{\mathsf{E}}
\newtheorem{lemme}{Lemma}

\begin{document}

\title{On Wilks' joint moment formulas for embedded principal minors of Wishart random matrices}

\author[1]{C. Genest}
\author[1]{F. Ouimet}
\author[2]{D. Richards}

\authormark{GENEST \textsc{et al.}}

\titlemark{Wilks' joint moment formulas}

\address[1]{\orgdiv{Department of Mathematics and Statistics}, \orgname{McGill University, Montr\'eal}, \orgaddress{\state{QC}, \country{Canada}}}

\address[2]{\orgdiv{Department of Statistics}, \orgname{Penn State University, University Park}, \orgaddress{\state{PA}, \country{USA}}}

\corres{Corresponding author: Christian Genest\email{christian.genest@mcgill.ca}}

\abstract[Abstract]{In 1934, the American statistician Samuel S.\ Wilks derived remarkable formulas for the joint moments of embedded principal minors of sample covariance matrices in multivariate Gaussian populations, and he used them to compute the moments of sample statistics in various applications related to multivariate linear regression. These important but little-known moment results were extended in 1963 by the Australian statistician A. Graham Constantine using Bartlett's decomposition. In this note, a new proof of Wilks’ results is derived using the concept of iterated Schur complements, thereby bypassing Bartlett’s decomposition. Furthermore, Wilks' open problem of evaluating joint moments of disjoint principal minors of Wishart random matrices is related to the Gaussian product inequality conjecture.}

\keywords{Gaussian product inequality, moments, principal minors, Schur complement, Wishart distribution, zonal polynomials}

\jnlcitation{\cname{%
\author{Genest C},
\author{Ouimet F}, and
\author{Richards D}}.
\ctitle{On Wilks' joint moment formulas for embedded principal minors of Wishart random matrices.} \cjournal{\it Stat} \cvol{2024;00(00):1--6}.}

\maketitle

\section{Introduction}\label{sec:1}

The Wishart distribution, named after the famous Scottish mathematician and agricultural statistician John Wishart (1898--1956), is a matrix-variate extension of the Gamma distribution which is commonly used in multivariate analysis. It is parameterized by a symmetric positive definite scale matrix $\Sigma$ of size $p \times p$ and a real $\alpha$ belonging to the Gindikin set $\{ 0, 1, \ldots , p - 1 \} \cup  (p - 1, \infty)$; see \citet{Gindikin1975}. When the degree-of-freedom parameter $\alpha = N$ is a positive integer, a random matrix $\mathfrak{X}$ with Wishart distribution decomposes as $\mathbf{Z}_1 \mathbf{Z}_1^{\top} + \cdots + \mathbf{Z}_N \mathbf{Z}_N^{\top}$, where $\mathbf{Z}_1,\ldots,\mathbf{Z}_N$ are mutually independent and identically distributed centered Gaussian random vectors of length $p$ with covariance matrix $\Sigma$.

While investigating inferential issues related to multivariate analysis of variance and regression problems, the renowned American statistician Samuel S.\ Wilks (1906--1964) derived a general expression for the joint moments of embedded principal minors of Wishart random matrices in the case in which the degree-of-freedom parameter $\alpha$ is integer-valued. This expression, Eq.~(2.59) in~\cite{Wilks1934} with $t = n$ and all the $\beta$ coefficients therein set equal to $0$, was later extended by \cite{Constantine1963} to the case of any real $\alpha > p-1$.

This remarkable but little-known result, stated in Theorem~\ref{thm:1} below in its general form, was proved by Wilks in the case of integer-valued degrees of freedom using moment-generating operators that commute with the integral sign in conjunction with Cauchy's expansion for the determinant of a bordered matrix and Jacobi's identity for complementary minors. See Section~0.8.5 of the book by \cite{Horn2013} for modern references to these notions.

In the hope of drawing renewed attention to this finding, the present note describes a new approach to the proof of Theorem~\ref{thm:1} based on iterated Schur complements, which might be of independent interest. Definitions and a preliminary lemma are provided in Sections~\ref{sec:2}~and~\ref{sec:3}, respectively. Theorem~\ref{thm:1} is stated and proved in Section~\ref{sec:4}. In Section~\ref{sec:5}, Wilks' open problem of evaluating the joint moments of disjoint principal minors of Wishart random matrices is restated and further motivated by its relationship with the Gaussian product inequality conjecture described, e.g., by \cite{Genest2024}.

\section{Definitions}\label{sec:2}

For any integer $p \in \mathbb{N} = \{1, 2, \ldots \}$, let $\mathcal{S}_{++}^p$ be the cone of real symmetric positive definite matrices of size $p\times p$. For any real number $\beta > (p-1)/2$, let
\[
\Gamma_p (\beta) = \int_{\mathcal{S}_{++}^p} |X|^{\beta - (p + 1)/2} \etr(-X) \mathrm{d} X
\]
denote the multivariate gamma function; see, e.g., Section~35.3 of the handbook by \cite{Olver2010}. Here, the notation $\mathrm{etr}(\cdot)$ stands for the exponential trace and $|\cdot|$ the determinant.

For any degree-of-freedom parameter $\alpha \in (p-1,\infty)$ and any scale matrix $\Sigma \in \mathcal{S}_{++}^p$, the probability density function of the $\mathrm{Wishart}_p (\alpha,\Sigma)$ distribution is given, for all matrices $X \in \mathcal{S}_{++}^p$, by
\[
f_{\alpha,\Sigma}(X) = \frac{|X|^{\alpha/2 - (p + 1)/2} \etr\{-(\Sigma^{-1}/2) X\}}{|2 \Sigma|^{\alpha/2} \Gamma_p(\alpha/2)}.
\]
Whenever a random matrix $\mathfrak{X}$ of size $p\times p$ follows this distribution, one writes $\mathfrak{X}\sim \mathcal{W}_p(\alpha,\Sigma)$ for short. For additional information concerning the Wishart distribution, see, e.g., Section~3.2.1 of the book by~\cite{Muirhead1982}.

\section{Preliminary lemma}\label{sec:3}

This section contains a well-known lemma that will be used in the proof of Theorem~\ref{thm:1}. Note that the proof of this lemma does not use Bartlett's decomposition in any way.

Items~(a)~and~(b) of Lemma~\ref{lem:1} below present the marginal distributions of Wishart random matrices; they can be found for example in Theorem~3.2.10 of the book by~\cite{Muirhead1982}. Item~(c) provides a general expression for the expectation of any principal minor of a Wishart random matrix. It is known in the Gaussian context since the work of~\cite{Wilks1932} on generalizations in the analysis of variance. The proof follows straightforwardly from item~(a) and a renormalization of the Wishart density.

\begin{lemme}
\label{lem:1}
Let $p_1,p_2 \in \mathbb{N}$ and set $p = p_1 + p_2$. Let $\alpha \in (p-1,\infty)$ and $\Sigma \in \mathcal{S}_{++}^p$. Assume that
\[
\mathfrak{X}
= \begin{bmatrix}
\mathfrak{X}_{11} & \mathfrak{X}_{12} \\
\mathfrak{X}_{21} & \mathfrak{X}_{22} \\
\end{bmatrix} \sim \mathcal{W}_p(\alpha,\Sigma), \quad \text{with}~~ \Sigma =
\begin{bmatrix}
\Sigma_{11} & \Sigma_{12} \\
\Sigma_{21} & \Sigma_{22} \\
\end{bmatrix}.
\]
Define $\mathfrak{X}/\mathfrak{X}_{11} = \mathfrak{X}_{22} - \mathfrak{X}_{21} \mathfrak{X}_{11}^{-1} \mathfrak{X}_{12}$ and $\Sigma/\Sigma_{11} = \Sigma_{22} - \Sigma_{21} \Sigma_{11}^{-1} \Sigma_{12}$ to be Schur complements of $\mathfrak{X}_{11}$ in $\mathfrak{X}$ and $\Sigma_{11}$ in $\Sigma$, respectively. Then
\begin{enumerate}
\item[(a)]
$\mathfrak{X}_{11}\sim \mathcal{W}_{p_1}(\alpha,\Sigma_{11})$;
\item[(b)]
$\mathfrak{X}/\mathfrak{X}_{11} \sim \mathcal{W}_{p_2}(\alpha - p_1, \Sigma/\Sigma_{11})$ and is independent of $(\mathfrak{X}_{11},\mathfrak{X}_{12})$;
\item[(c)]
For every real $\nu_1 \in [0,\infty)$, $\EE(|\mathfrak{X}_{11}|^{\nu_1}) = |2 \Sigma_{11}|^{\nu_1} {\Gamma_{p_1}(\alpha/2 + \nu_1)} / {\Gamma_{p_1}(\alpha/2)}$.
\end{enumerate}
\end{lemme}

\section{New proof of an old theorem}\label{sec:4}

Theorem~\ref{thm:1} below extends Wilks' original expression for the joint moments of embedded principal minors of Wishart random matrices to the case in which the degree-of-freedom parameter $\alpha$ is allowed to be non-integer. This version of the theorem was proved by~\cite{Constantine1963} through a reproductive property of the so-called zonal polynomials under Wishart expectations; see his Theorem~1. Constantine's proof relied on a technique called Bartlett's decomposition after \cite{Bartlett1934}, where the joint distribution of the entries of the lower triangular matrix is known in the Cholesky decomposition of a Wishart random matrix; see, e.g., Section~3.2.4 of the book by~\cite{Muirhead1982}.

\begin{theorem}
\label{thm:1}
Let $d, p_1, \ldots, p_d \in \mathbb{N}$ and set $p = p_1 + \cdots + p_d$. Let $\alpha \in (p - 1, \infty)$ and $\Sigma \in \mathcal{S}_{++}^p$. Assume that
\[
\mathfrak{X} =
\begin{bmatrix}
\mathfrak{X}_{11} & \cdots & \mathfrak{X}_{1d} \\
\vdots & \ddots & \vdots \\
\mathfrak{X}_{d1} & \cdots & \mathfrak{X}_{dd} \\
\end{bmatrix}
\sim \mathcal{W}_p(\alpha,\Sigma)
\quad \text{with}~~ \Sigma =
\begin{bmatrix}
\Sigma_{11} & \cdots & \Sigma_{1d} \\
\vdots & \ddots & \vdots \\
\Sigma_{d1} & \cdots & \Sigma_{dd} \\
\end{bmatrix}.
\]
For any integer $i \in \{1,\ldots,d\}$, define $\mathfrak{X}_{1:i,1:i}$ to be the principal submatrix of $\mathfrak{X}$ composed of the blocks $(\mathfrak{X}_{k\ell})_{1\leq k,\ell\leq i}$. Then, for all reals $\nu_1, \ldots, \nu_d \in [0,\infty),$
\[
\EE\!\left(\prod_{i=1}^d |\mathfrak{X}_{1:i,1:i}|^{\nu_i}\right)
= \prod_{i=1}^d |2 \Sigma_{1:i,1:i}|^{\nu_i} \frac{\Gamma_{p_i} \big(\alpha/2 - \sum_{k=1}^{i-1} p_k/2 + \sum_{k=i}^d \nu_k \big)}{\Gamma_{p_i} \big( \alpha/2 - \sum_{k=1}^{i-1} p_k/2 \big)},
\]
with the convention that a sum over an empty set equals zero.
\end{theorem}

\begin{proof}
Let $\mathfrak{X}_{2:d,2:d}$, $\mathfrak{X}_{2:d,1}$ and $\mathfrak{X}_{1,2:d}$ be the submatrices of $\mathfrak{X}$ formed by the blocks $\smash{(\mathfrak{X}_{k\ell})_{2\leq k,\ell\leq d}}$, $\smash{(\mathfrak{X}_{k1})_{2\leq k\leq d}}$ and $\smash{(\mathfrak{X}_{1\ell})_{2\leq \ell\leq d}}$, respectively. Define
\[
\mathfrak{X}^{[1]} = \mathfrak{X} / \mathfrak{X}_{11} = \mathfrak{X}_{2:d,2:d} - \mathfrak{X}_{2:d,1} \mathfrak{X}_{11}^{-1} \mathfrak{X}_{1,2:d}
\]
to be the Schur complement of the block $\mathfrak{X}_{11}$ in $\mathfrak{X}$. Similarly, the matrix $\mathfrak{X}^{[1]}$ is composed of $(d-1)\times (d-1)$ blocks, so the Schur complement of the top-left block in $\mathfrak{X}^{[1]}$ is defined by
\[
\mathfrak{X}^{[2]} = \mathfrak{X}^{[1]} / \mathfrak{X}^{[1]}_{11} = \mathfrak{X}^{[1]}_{2:(d - 1), 2:(d - 1)} - \mathfrak{X}^{[1]}_{2:(d-1),1} (\mathfrak{X}^{[1]}_{11})^{-1} \mathfrak{X}^{[1]}_{1,2:(d-1)}.
\]

This process of taking iterative Schur complements of the top-left block at each step goes on similarly. After $k\in \{1,\ldots,d-1\}$ steps, one obtains the block matrix
\[
\mathfrak{X}^{[k]} = \mathfrak{X}^{[k-1]} / \mathfrak{X}^{[k-1]}_{11} = \mathfrak{X}^{[k-1]}_{2:(d - k + 1), 2:(d - k + 1)} - \mathfrak{X}^{[k-1]}_{2:(d-k+1),1} (\mathfrak{X}^{[k-1]}_{11})^{-1} \mathfrak{X}^{[k-1]}_{1,2:(d-k+1)},
\]
and one writes $\smash{\mathfrak{X}^{[k]}_{11}}$ for its top-left block.

The same process can also be applied, beginning with a principal submatrix of $\mathfrak{X}$ instead of the full matrix. For example, for any given integer $i \in \{1,\ldots,d\}$, take $\mathfrak{X}_{1:i,1:i}$ to be the principal submatrix of $\mathfrak{X}$ composed of the blocks $(\mathfrak{X}_{k\ell})_{1 \leq k, \ell \leq i}$. By applying exactly the same $k \in \{1,\ldots,i-1\}$ steps of taking iterative Schur complements of the top-left block at each step, one can naturally express the result as $(\mathfrak{X}_{1:i,1:i})^{[k]}$, and its top-left block as $\smash{(\mathfrak{X}_{1:i,1:i})^{[k]}_{11}}$.

Fortunately, a simplification occurs due to the quotient property of Schur complements; see, e.g., Eq.~(0.8.5.12) of the book by~\cite{Horn2013}. Indeed, one has the convenient result that the top-left block of $(\mathfrak{X}_{1:i,1:i})^{[k]}$ corresponds with the top-left block of $\mathfrak{X}^{[k]}$, viz.,
\begin{equation}
\label{eq:1}
(\mathfrak{X}_{1:i,1:i})^{[k]}_{11} = \mathfrak{X}^{[k]}_{11}.
\end{equation}
Everything mentioned above also applies verbatim to $\Sigma$.

It is well-known from linear algebra that the determinant of a block matrix can be written as the product of the determinant of the top-left block and the determinant of its Schur complement; see, e.g., Eq.~(0.8.5.1) of the book by~\cite{Horn2013}. Specifically, if
\[
M =
\begin{bmatrix}
M_{11} & M_{12} \\
M_{21} & M_{22}
\end{bmatrix},\vspace{-2mm}
\]
then
\[
|M| = |M_{11}| \times |M^{[1]}_{11}| =  |M^{[0]}_{11}| \times |M^{[1]}_{11}|,
\]
with the convention that $M^{[0]} = M$. By iterating this last formula and applying the identity \eqref{eq:1}, one finds that, for every integer $i \in \{1,\ldots,d\}$,
\[
|\mathfrak{X}_{1:i,1:i}| = \prod_{k=1}^i |(\mathfrak{X}_{1:i,1:i})^{[k-1]}_{11}| = \prod_{k=1}^i |\mathfrak{X}^{[k-1]}_{11}|.
\]
Hence, for all reals $\nu_1, \ldots, \nu_d \in [0,\infty),$
\begin{equation}
\label{eq:2}
\prod_{i=1}^d |\mathfrak{X}_{1:i,1:i}|^{\nu_i}
= \prod_{i=1}^d \prod_{k=1}^i |\mathfrak{X}^{[k-1]}_{11}|^{\nu_i}
= \prod_{k=1}^d \prod_{i=k}^d |\mathfrak{X}^{[k-1]}_{11}|^{\nu_i}
= \prod_{k=1}^d |\mathfrak{X}^{[k-1]}_{11}|^{V_k},
\end{equation}
where, for each integer $k \in \{ 1, \ldots, d\}$, $V_k = \nu_k + \dots + \nu_d$ and the notation $P_k = p_1 + \dots + p_k$ is used below with the convention that $P_0 = 0$.

Taking the expectation in the last equation and applying Lemma~\ref{lem:1} iteratively, one finds successively
\[
\begin{aligned}
\EE\!\left(\prod_{i=1}^d |\mathfrak{X}_{1:i,1:i}|^{\nu_i}\right)
&= \EE\!\left(\prod_{k=1}^d |\mathfrak{X}^{[k-1]}_{11}|^{V_k}\right) \\
&= 2^{p_1 V_1} |\Sigma^{[0]}_{11}|^{V_1} \frac{\Gamma_{p_1}(\alpha/2 - P_0/2 + V_1)}{\Gamma_{p_1}(\alpha/2 - P_0/2)} \times \EE\!\left(\prod_{k=2}^d |\mathfrak{X}^{[k-1]}_{11}|^{V_k}\right) \\
&= 2^{p_1 V_1} |\Sigma^{[0]}_{11}|^{V_1} \frac{\Gamma_{p_1}(\alpha/2 - P_0/2 + V_1)}{\Gamma_{p_1}(\alpha/2 - P_0/2)} \\
&\qquad\times 2^{p_2 V_2} |\Sigma^{[1]}_{11}|^{V_2} \frac{\Gamma_{p_2}(\alpha/2 - P_1/2 + V_2)}{\Gamma_{p_2}(\alpha/2 - P_1/2)} \\
&\qquad\times \EE\!\left(\prod_{k=3}^d |\mathfrak{X}^{[k-1]}_{11}|^{V_k}\right) \\
&= \cdots \\
&= \prod_{k=1}^d 2^{p_k V_k} |\Sigma^{[k-1]}_{11}|^{V_k} \frac{\Gamma_{p_k}(\alpha/2 - P_{k-1}/2 + V_k)}{\Gamma_{p_k}(\alpha/2 - P_{k-1}/2)}.
\end{aligned}
\]
Using \eqref{eq:2} this time for the matrix $\Sigma$, and reorganizing some factors using the summation by parts relationship
\[
\smash{\sum_{k=1}^d p_k V_k = \sum_{i=1}^d \nu_i P_i},
\]
one obtains
\[
\prod_{i=1}^d 2^{\nu_i P_i} |\Sigma_{1:i,1:i}|^{\nu_i} \frac{\Gamma_{p_i}(\alpha/2 - P_{i-1}/2 + V_i)}{\Gamma_{p_i}(\alpha/2 - P_{i-1}/2)}.
\]
This concludes the proof.
\end{proof}

\section{An open problem}\label{sec:5}

Following Theorem~\ref{thm:1}, a natural question, which Wilks also had the foresight to ask in 1934, is whether one can derive a general formula for the joint moments of disjoint principal minors of Wishart random matrices, viz.,
\begin{equation}
\label{eq:3}
\EE\!\left(\prod_{i=1}^d |\mathfrak{X}_{ii}|^{\nu_i}\right).
\end{equation}
When $\alpha\in (p-1,\infty)$ and the scale matrix $\Sigma$ is block-diagonal, this problem is simple to solve using the independence of the diagonal blocks in conjunction with Lemma~\ref{lem:1}~(c). However, as~\citet[p.~326]{Wilks1934} pointed out, the general case is ``extremely complicated.'' The problem is still open to the present day.

One motivation for obtaining a formula for \eqref{eq:3} is its potential utility in tackling the Wishart analog of the Gaussian product inequality conjecture, proposed by~\cite{Genest2024}. The conjecture is stated below for convenience.

\begin{conjecture}
\label{conj:1}
Suppose that the assumptions of Theorem~\ref{thm:1} hold. Then, for all reals $\nu_1, \ldots, \nu_d \in [0,\infty),$
\[
\EE\!\left(\prod_{i=1}^d |\mathfrak{X}_{ii}|^{\nu_i}\right) \geq \prod_{i=1}^d \EE\!\left(|\mathfrak{X}_{ii}|^{\nu_i}\right).
\]
\end{conjecture}

The Gaussian product inequality conjecture itself, which recently gained a lot of traction, states that if $\mathbf{Z} = (Z_1,\ldots,Z_d)$ is a centered Gaussian random vector with covariance matrix $\Sigma$, then for all reals $\nu_1, \ldots, \nu_d\in [0,\infty)$,
\[
\EE \left( \prod_{i=1}^d |Z_i|^{2 \nu_i} \right) \geq \prod_{i=1}^d \EE \big (|Z_i|^{2 \nu_i} \big).
\]
The special case $\nu_1 = \dots = \nu_d\in \mathbb{N}_0 = \{0, 1, \ldots \}$ corresponds to the real polarization problem in functional analysis, originally formulated by \cite{Benitez1998}. It was reframed in the above Gaussian setting by~\cite{Frenkel2008}. Readers may refer to \cite{Genest2024} for a brief survey of currently known results and to \cite{Herry2024} for the latest advancement.


\bmsection*{AUTHOR CONTRIBUTIONS}

The manuscript conceptualization, writing, reviewing, and editing were shared equally by Christian Genest, Fr\'ed\'eric Ouimet, and Donald Richards. The mathematical arguments were mostly developed by Fr\'ed\'eric Ouimet.

\bmsection*{CONFLICT OF INTEREST STATEMENT}

The authors have no conflicts of interest to report.

\bmsection*{DATA AVAILABILITY STATEMENT}

Not applicable.

\bmsection*{Financial disclosure}

None reported.

\bmsection*{ORCID}

Christian Genest: \orcidlink{https://orcid.org/0000-0002-1764-0202} \url{https://orcid.org/0000-0002-1764-0202} \\
Fr\'ed\'eric Ouimet: \orcidlink{https://orcid.org/0000-0001-7933-5265} \url{https://orcid.org/0000-0001-7933-5265} \\
Donald Richards: \orcidlink{https://orcid.org/0000-0002-2254-4514} \url{https://orcid.org/0000-0002-2254-4514}

\bibliography{GeOuRiStat}

\bmsection*{Author Biography}

\begin{biography}{}{{\textbf{Christian Genest} is Professor and Canada Research Chair in Stochastic Dependence Modeling at McGill University. His research is at the crossroad between extreme-value analysis, multivariate analysis, and nonparametric statistics.

\noindent
\textbf{Fr\'ed\'eric Ouimet} is a Postdoctoral Fellow at McGill University. His research centers on asymptotic theory, multivariate analysis, nonparametric density estimation, Gaussian fields, and branching processes.

\noindent
\textbf{Donald Richards} is a Distinguished Professor Emeritus of Statistics at Penn State University. His research interests include multivariate statistical analysis, probability inequalities, and special functions of matrix argument.}}
\end{biography}

\end{document}